\documentclass{amsart}
\usepackage{amssymb}
\usepackage{hyperref}

\theoremstyle{plain}
\newtheorem{theorem}{Theorem}
\newtheorem{proposition}[theorem]{Proposition}
\newtheorem{lemma}[theorem]{Lemma}
\newtheorem{corollary}[theorem]{Corollary}

\begin{document}
\title[Volume entropy of a family of rank one, split-solvable Lie groups]{Volume entropy of a family of rank one, split-solvable Lie groups of Abelian type}
\author{René I. García-Lara}
\address{Institute of Mathematics\\
    UNAM\\
    Cuernavaca Unit,
    Cuernavaca, México.}
\email{israel.garcia@im.unam.mx}

\begin{abstract}
    We study a family of metrics on Euclidean space that generalize the left-invariant metric of the SOL group and the metric of the logarithmic model of Hyperbolic space. Suppose G is a connected, simply-connected, Heintze group of Abelian type with diagonalizable derivation or the horospherical product of two such groups. In this scenario, G is isometric to Euclidean space with a metric of the type considered. We have derived a formula for the volume entropy of metrics in this family and used it to solve a conjecture related to a family of 3-manifolds that interpolates between the SOL group and hyperbolic space.
\end{abstract}

\date{\today}
\maketitle

\section{Statement of the problem and main results}

Let \(G\) be a connected, simply-connected, solvable Lie group with Lie Algebra \(\mathfrak{g}\), such that for every \(X \in \mathfrak{g}\), the spectrum of the adjoin operator, \(\operatorname{ad}(X): \mathfrak{g} \to \mathfrak{g}\), has no purely imaginary eigenvalue. It is well known that under this condition, the 
exponential map \(\mathfrak{g} \to G\) is a diffeomorphism~\cite[Thm.~5.2.16]{fujiwaraHarmonicAnalysisExponential2015}, if we also assume that the derived algebra, \(\mathfrak{g}' = [\mathfrak{g}, \mathfrak{g}]\), is Abelian and  satisfies \(\dim \mathfrak{g}' = \dim \mathfrak{g} - 1\), then for any \(A \in \mathfrak{g} \setminus\mathfrak{g}'\), there is a Lie group isomorphism of \(G\) onto \(G(A) = N \rtimes_{A} \mathbb{R}\), where \(N = \exp(\mathfrak{g}') \cong \mathbb{R}^{N}\) and the group operation on \(G(A)\) is,
\begin{equation}\label{eq:group-structure}
    (x, t) * (x', t') = (x + \exp(t\,A)x', t + t').
\end{equation}
\(G(A)\) is an example of Abelian split-solvable Lie group, recall each inner product on \(\mathfrak{g}\) determines a left-invariant, Riemannian metric on \(G(A)\), and assume given a pair, of an inner product \(g\), such that \(A\) is orthogonal to \(\mathfrak{g}'\), and an orthonormal basis \( \{e_{1}, \ldots, e_{N+1}\} \subset \mathfrak{g}\) such that \(e_{N+1} \in \mathfrak{g}\setminus \mathfrak{g}'\) and \(\operatorname{ad}(A)|_{\mathfrak{g}'}\) has a Jordan block diagonal form in the basis \( \{e_{1}, \ldots, e_{N}\} \). Under these conditions, there is  a vector \(a \in \mathbb{R}^{N}\) such that in the coordinate system \(\mathbb{R}^{N+1} \to G(A)\), \(x \mapsto \sum x_{i}e_{i}\), the left invariant metric is 
\begin{equation}\label{eq:the-horospheric-metric}
    g_{a} = \sum_{i=1}^{N}e^{-2a_{i}x_{N+1}}dx_{i}^2 + dx_{N+1}^2.
\end{equation} 
Our aim is to study the growth of geodesic balls in \((\mathbb{R}^{N+1}, g_{a})\) for a generic vector \(a\), regardless of the underlying algebraic structure. Let \(a \in \mathbb{R}^N\), if we define \(A = \operatorname{diag}(a) \in \mathrm{Mat}(n\times n)\), then the space \((\mathbb{R}^{N+1}, g_{a})\) is isometric to \(G(A)\), whence, it also is homogeneous, in particular, if \(B(x, r)\) denotes the geodesic ball of center \(x\) and radius \(r\), then \(\operatorname{Vol}(B(x,r)) = \operatorname{Vol}(B(0, r))\). In order to
study the rate of growth of geodesic balls, 
recall the volume entropy of a Riemannian manifold is the limit
\begin{equation}\label{eq:volume-entropy-limit}
    \delta(g_{a}, x) = \lim_{r \to \infty} \frac{\operatorname{Vol}(B(x, r))}{r},
\end{equation}
since the space is homogeneous, we can drop the dependency on the base point \(x\) and denote the volume entropy as \(\delta(g_{a})\). With these definitions at hand, the main result of the work is,
\begin{theorem}\label{thm:volume-entropy}
    Let \(a \in \mathbb{R}^{N}\), 
    the volume entropy of \((\mathbb{R}^{N+1}, g_{a})\) is 
    \begin{equation}\label{eq:volume-entropy-formula}
        \delta(g_{a}) = \max \left(
        \sum_{0 \leq a_{i}} a_{i}, \sum_{a_{i} < 0} |a_{i}|
        \right),
    \end{equation}
    where  
    the metric \(g_{a}\) is given by Equation~\eqref{eq:the-horospheric-metric}.
\end{theorem}
Theorem~\ref{thm:volume-entropy} has applications in differential geometry and analysis on Lie groups, let \(A = \mathrm{diag}(1, -1)\), then \(\mathbb{R}^2 \rtimes_{A} \mathbb{R}\) is the group \(\mathrm{SOL}\), one of the eight Thurston's geometries of three-dimensional space. A
detailed description of the geometry of \(\mathrm{SOL}\) was carried out by
Troyanov~\cite{troyanovHorizonSol1998} and further developed by
Coiculescu and Schwartz in~\cite{coiculescuSpheresSol2022,schwartzAreaGrowthSol2020}. Schwartz proves that the volume entropy of \(\mathrm{SOL}\) is \(1\), in accordance to Theorem~\ref{thm:volume-entropy}. In fact, the lower bound estimate on the growth of geodesic balls is based on Schwartz' ideas. Further developments on the geometry of the groups \(\mathbb{R}^2\rtimes_{A}\mathbb{R}\), seen as three-manifolds, was carried out by Coiculescu~\cite{coiculescuInterpolationSolHyperbolic2022}, who studied the geometry of these spaces for \(A = \mathrm{diag}(1, -\alpha)\), where the parameter \(\alpha\in [-1, 1]\)  determines an interpolation between three Thurston geometries, namely, \(\mathrm{SOL}\), \(\mathbb{H}^2\times \mathbb{R}\) and \(\mathbb{H}^3\) in the cases \(\alpha = 1, 0\) and \(-1\), respectively. It was conjectured that the volume entropy of \(\mathbb{R}^{2} \rtimes_{A} \mathbb{R}\) is a monotonically decreasing function of \(\alpha \), however, 
in view of Theorem~\ref{thm:volume-entropy}, we found the following correction
to the conjecture,
\begin{corollary}
    Let \(a = (1, -\alpha) \in \mathbb{R}^2\), then for the metric \(g_{a}\) defined as in~\eqref{eq:the-horospheric-metric}, 
    the volume entropy of \((\mathbb{R}^{3}, g_{a})\) is
    \[
        \delta(g_{\alpha}) =
        \begin{cases}
            1 - \alpha, & \text{if } \alpha < 0,        \\
            1,            & \text{if } \alpha \in [0, 1], \\
            \alpha        & \text{if } \alpha > 1.
        \end{cases}
    \]
\end{corollary}
If \(A \in \mathrm{Mat}(N\times N)\) only has eigenvalues of positive real part, then \(G(A)\) is a particular case of Abelian Heintze group~\cite{heintzeHomogeneousManifoldsNegative1974}, where the adjective Abelian refers to the fact that the maximal, nilpotent, subgroup \(N \triangleleft \, G(A)\) is Abelian. It is well known that Heintze groups admit a left-invariant, hyperbolic metric, not necessarily the metric determined by \(A\), however, since \(A\) is diagonalizable, we will see in the next section that as a consequence of Proposition~\ref{prop:sectional-curvature-formula}, the left invariant metric~\eqref{eq:the-horospheric-metric} is indeed hyperbolic, moreover, as a second application of the Theorem, we find the following corollary, 
\begin{corollary}\label{cor:heintze-volume-entropy}
    If \(A \in \mathrm{Mat}(n\times n)\) is diagonalizable in \(\mathbb{C}\) and such that each eigenvalue of \(A\) has positive real part, then the volume entropy of the Abelian Heintze group \(G(A)\) is 
    \[\delta(G(A)) = \operatorname{tr}(J),\]
    where \(J\) is the complex Jordan form of \(A\).
\end{corollary}
Given two Heintze groups, \(G(A)\), \(G(B)\), not necessarily of Abelian type, define the horospherical product~\cite[Pg.~2]{ferragutGeodesicsVisualBoundary}, 
\begin{equation}\label{eq:horospheric-product-def}
    G(A) \bowtie G(B) = (N_1 \times N_{2}) \rtimes_{\operatorname{diag}(A, -B)} \mathbb{R},
\end{equation}
where \(\operatorname{diag}(A, -B)\) is the block-diagonal matrix with matrix blocks \(A\), \(-B\). 
Horospherical products where introduced by Woess in~\cite{woessWhatHorocyclicProduct2013} as a graph and group theoretic structure, later, Ferragut~\cite{ferragutGeodesicsVisualBoundary,
ferragutGeometricRigidityQuasiisometries} extended the product to proper, geodesically complete, Gromov hyperbolic, Busemann spaces and named the product horospherical. The study of horospherical products is an area of active research both on the algebraic and analytic side, as can be seen in~\cite{eskinQuasiisometricRigiditySolvable2011,
    pengCoarseDifferentiationQuasiisometries2011} and the references therein. Horospherical products generalize \(\mathrm{SOL}\) to more arbitrary spaces, in fact, \(\mathrm{SOL} \cong \mathbb{H}^2 \bowtie \mathbb{H}^{2}\). If \(G(A)\) and \(G(B)\) are Heintze groups of Abelian type and the matrices \(A\), \(B\) are diagonalizable, Theorem~\ref{thm:volume-entropy} and Corollary~\ref{cor:heintze-volume-entropy} imply the following,
\begin{corollary}\label{cor:horospherical-product-volume-entropy}
    If \(G(A)\) and \(G(B)\) are two Heintze groups of Abelian type, such that \(A\) and \(B\) are diagonalizable matrices, then the horospherical product~\eqref{eq:horospheric-product-def} has volume entropy,
    \[\delta(G(A)\bowtie G(B)) = \max(\delta(G(A)), \delta(G(B))),\]
    where the volume entropies \(\delta(G(A))\) and \(\delta(G(B))\) are given in Corollary~\ref{cor:heintze-volume-entropy}.
\end{corollary}
Corollary~\ref{cor:horospherical-product-volume-entropy} says that even tough horospherical products are not hyperbolic, as we will see when we compute the curvature of the metric~\eqref{eq:the-horospheric-metric} in Section~\ref{sec:geometry-of-the-metric}, the volume of geodesic balls grows exponentially fast.

 To conclude this section, we describe the contents of the
article. Section~\ref{sec:geometry-of-the-metric} is about the geometry of the metric~\eqref{eq:the-horospheric-metric}, we prove a formula for the sectional curvature in this section. In Proposition~\ref{prop:curvature-bounds} we find some general bounds for the sectional curvature which complement the known bounds for the sectional curvature of \(\mathrm{SOL}\). The logarithmic-hyperbolic model of hyperbolic space is introduced in~\ref{sec:log-hyp-model}, then in sections~\ref{sec:lower-bound} and~\ref{sec:upper-bounds}, we use the geometry of hyperbolic space to estimate the volume of geodesic balls when the metric \(g_{a}\) is hyperbolic. Finally, we prove Theorem~\ref{thm:volume-entropy} in Section~\ref{sec:volume-growth-general-case}, finding an upper bound for the volume of geodesic balls is relatively easy, however, the lower bound requires more effort. In~\cite{schwartzAreaGrowthSol2020}, Schwartz found a lower bound in the case of \(\mathrm{SOL}\) by a thorough analysis of geodesic spheres, which relies on the study of the cut locus. In the general case this approach is unfeasible, instead, in Section~\ref{sec:lower-bound-general}, we  conclude the proof of the Theorem using an inductive argument, based on the hyperbolic case we previously proved.

\section{Geometry of the metric}\label{sec:geometry-of-the-metric}

In this section we describe necessary properties of geodesic equations. 
It is not difficult to see, for example by the variational method, that a curve \(\gamma: \mathbb{R} \to \mathbb{R}^{N+1}\) is a geodesic of the metric~\eqref{eq:the-horospheric-metric} if and only if
their components \(\gamma_i(s)\) satisfy the equations
\begin{align}\label{eq:geodesics}
    \dot{\gamma}_i  = C_{i}e^{2a_i\gamma_{N+1}}, &&
    \ddot{\gamma}_{N+1} = - \sum_{i=1}^N a_{i}C_i^2 e^{2a_i\gamma_{N+1}},    
\end{align}
where \( C_{1}, \ldots, C_{N} \), are constants we can assume satisfy the condition \(\sum C_i^2 < 1\), provided \(\gamma \) is parametrized by arc-length.
From the geodesic equations, we deduce hyperplanes \(x_{i} = const\).\ are totally geodesic, moreover, if \(C_{i} > 0\) (resp. \(C_{i} < 0\)) and \(\gamma(0) = 0\), then for \(s > 0 \), \(\gamma(s)\) remains in the half-space of points \(x\) such that \(x_{i} > 0\) (resp.\  \(x_{i} < 0\)). Furthermore, we remark that if \(a_{i} > 0\) for all \(i\) (resp.\  \(a_{i} < 0\)), then \(\dot{\gamma}_{N+1}\) is monotonous decreasing (resp.\ increasing).
As the following proposition shows, unless all \(a_{i}\) are non-null and of the same sign, horospherical products have tangent two-planes of both positive and negative sectional curvature.
\begin{proposition}\label{prop:sectional-curvature-formula}
    Let \(a \in \mathbb{R}^{N}\) be the vector of coefficients of the metric \(g_{a}\) defined in Equation~\eqref{eq:the-horospheric-metric}, and let \( \{E_{1}, \ldots, E_{N+1}\} \) be the canonical basis of \(T_{0}\mathbb{R}^{N+1}\), if \(P = \lambda E_{N+1} + \mu X \), \(Q = Y\), where \(\lambda^2 + \mu^2 = 1\), and \(X, Y\) are vectors in the tangent hyperplane \(E_{N+1} = 0\) such that \(P\) and \(Q\) are orthonormal, then the sectional curvature of the 2-plane spanned by \( \{P, Q\} \) is,
    \[
        \kappa = -\lambda^2 \sum a_{i}^2Y_{i}^2
        -\mu^2 \left(\sum a_i X_{i}^2\cdot \sum a_{i} Y_{i}^2
        - {\left(\sum a_{i} X_{i}Y_{i}\right)}^2\right),
    \]
    in particular, if all the coordinates \(a_{i}\) are of the same sign, \((\mathbb{R}^{N+1}, g_{a})\) is of negative, pinched sectional curvature.
\end{proposition}
If all the coordinates of \(a\) are positive, the group structure on \(\mathbb{R}^{N+1}\) corresponds to a real Heintze group and is well known that the curvature is negative~\cite{heintzeHomogeneousManifoldsNegative1974}, 
in fact, Heintze groups are the only known examples of solvable Lie groups of negative curvature~\cite{ferragutGeodesicsVisualBoundary}.
\begin{proof}
    By an abuse of notation, we also denote by \(E_{i}\) the left invariant vector field generated by the canonical vector \(E_{i}|_{0} \in T_0\mathbb{R}^{N+1}\). Recall the construction of the left invariant metric~\eqref{eq:the-horospheric-metric} relies on the fact that each \(E_{i}\), \(i \in \{1, \ldots, N\} \), is an eigenvector of the adjoint operator of \(E_{N+1}\) with eigenvalue \(a_{i}\), i.e.,
    \[
        [E_{N+1}, E_{i}] = a_{i} E_{i}, \qquad i \in \{1, \ldots, N\}.
    \]
    Since the nilpotent subspace \(E_{N+1} = 0\) is abelian, computation of the structure constants simplify and since the vector fields \(E_{i}\) are orthonormal, the following expression for the Koszul formula holds,
    \[
        \langle \nabla_{E_{i}}E_{j}, E_{k} \rangle = \frac{1}{2} \left(
        \langle [E_{i}, E_{j}], E_{k} \rangle
        - \langle [E_{i}, E_{k}], E_{j} \rangle
        - \langle [E_{j}, E_{k}], E_{i} \rangle
        \right).
    \]
    From this formula, we deduce the following identities,
    \begin{equation}\label{eq:connetion-identities}
        \begin{aligned}
            \nabla_{E_{i}}E_{j} = a_{i}\delta_{ij} E_{N+1}, &  &
            \nabla_{E_{N+1} E_{N+1}} = 0,                        \\
            \nabla_{E_{i}} E_{N+1} = -a_{i} E_{i},          &  &
            \nabla_{E_{N+1}} E_{i} = 0,
        \end{aligned}
    \end{equation}
    where \(\delta_{ij}\) is Kronecker's delta and the indices run from \(1\) to \(N\).
    For the following computations, recall the curvature tensor has the formula,
    \[
        R(X, Y, Z, W) = \langle \nabla_{X}\nabla_{Y} Z - \nabla_{Y}\nabla_{X} Z - \nabla_{[X, Y]} Z, W \rangle.
    \]
    Let \(\Pi \subset T_0\mathbb{R}^{N+1}\) be the tangent hyperplane spanned by \( \{E_{1}, \ldots, E_{N}\} \), if \(X, Y \in \Pi \), the connection identities~\eqref{eq:connetion-identities} imply
    \begin{align}\label{eq:connection-formulas-hyperplane}
        \nabla_{X}Y = \left(\sum X_{i}Y_{i} a_{i}\right) E_{N+1}, &  &
        \nabla_{X} E_{N+1} = -\sum a_{i}X_{i}\in \Pi,
    \end{align}
    furthermore, since \(X\) and \(Y\)  commute, we deduce that for any triple \(X, Y, Z \in \Pi \), \(R(X, Y, Z, E_{N+1}) = 0\).
    Since the vectors \(P, Q\) are orthonormal, the sectional curvature is,
    \begin{align*}
        \kappa & = R(P, Q, Q, P)                                                          \\
               & = R(\lambda E_{N+1} + \mu X, Y, Y, \lambda E_{N+1} + \mu X)              \\
               & = \lambda^2 R(E_{N+1}, Y, Y, E_{N+1}) + 2\lambda \mu R(X, Y, Y, E_{N+1})
        + \mu^2 R(X, Y, Y, X)                                                             \\
               & = \lambda^2 R(E_{N+1}, Y, Y, E_{N+1}) + \mu^2 R(X, Y, Y, X),
    \end{align*}
    in the third equation we used the multilinearity and symmetries of the curvature tensor to simplify the expression. Let \(X, Y \in \Pi \) be a pair of orthonormal vectors and let \(P, Q\) be defined as in the statement of the proposition, using the connection formulas~\eqref{eq:connection-formulas-hyperplane}, we find,
    \begin{align*}
        R(X, Y, Y, X) & = \langle
        \nabla_{X}\nabla_{Y} Y - \nabla_{Y} \nabla_{X} Y, X
        \rangle                                                           \\
                      & = -\left(\sum a_i X_{i}^2\cdot \sum a_{i} Y_{i}^2
        - {\left(\sum a_{i} X_{i}Y_{i}\right)}^2\right),
    \end{align*}
    likewise,
    \begin{align*}
        R(E_{N+1}, Y, Y, E_{N+1}) & = \langle
        \nabla_{E_{N+1}} \nabla_{Y} Y - \nabla_{Y} \nabla{E_{N+1}} Y - \nabla_{[E_{N+1}, Y]}Y, E_{N+1}
        \rangle                                                                        \\
                                  & = -\langle \nabla_{[E_{N+1}, Y]}Y, E_{N+1} \rangle \\
                                  & = - \sum a_{i}^2 Y_{i}^2.
    \end{align*}
    This proves the formula for the sectional curvature. For the last statement, if all the coordinates of \(a\) have the same sign, \(R(X, Y, Y, X) < 0\) by the Cauchy-Schwarz inequality,
    hence,
    \[
        \kappa = -\lambda^2 \sum a_{i}^2 Y_{i}^2 + \mu^2 R(X, Y, Y, X) < 0.
    \]
\end{proof}
Given the explicit formula for the pinched sectional curvature, we have the following,
\begin{proposition}\label{prop:curvature-bounds}
    Let \(a \in \mathbb{R}^N\), define, 
    \begin{align*}
        M = \max(a_{1}, \ldots, a_{N}) && \text{and} &&
        m = \min (a_{1}, \ldots, a_{N}),
    \end{align*}
    the sectional curvature satisfies the following bounds.
    \begin{enumerate}
        \item If all the coordinates \(a_{i}\) are of the same sign,
              \[-M^2 \leq \kappa \leq -m^2,\]
        \item else,
              \[- \max(M^2, m^2) \leq \kappa \leq -Mm. \]
    \end{enumerate}
\end{proposition}
\begin{proof}
    Let \(\Pi \subset T_{0}\mathbb{R}^{N+1}\) be the tangent, horizontal, hyperplane of vectors \(X\) such that \(X_{N+1} = 0\), and let \(X, Y \in \Pi \) be orthonormal, define \(P = \lambda E_{N+1} + \mu X\), with \(\lambda^2 + \mu^2 = 1\), as in the proof of Proposition~\ref{prop:sectional-curvature-formula},
    we have the basic identities
    \begin{align*}
        \sum a_i X_{i}^2\cdot \sum a_{i} Y_{i}^2
        - {\left(\sum a_{i} X_{i}Y_{i}\right)}^2     & =
        \sum_{i < j} a_{i}a_{j} {(X_{i} Y_{j} - X_{j} Y_{i})}^2, \\
        \sum_{i < j} {(X_{i} Y_{j} - X_{j} Y_{i})}^2 & = 1,
    \end{align*}
    whose proof can be found in Appendix~\ref{sec:appendix-cauchy-schwarz}.
    If all the coordinates \(a_{i}\) are either positive or negative, it is straightforward that \(m^2 \leq a_{i}a_{j} \leq M^2\), for \(i, j \in \{1, \ldots, N\} \), whence,
    \begin{align*}
        -\lambda^2 M^2 - \mu^2 M^2 \leq -\lambda^2 \sum a_{i}^2 Y_{i}^2 - \mu^2
        \sum_{i < j} a_{i}a_{j} {(X_{i}Y_{j} - X_{j}Y_{i})}^2
        \leq -\lambda^2 m^2 - \mu^2 m^2,
    \end{align*}
    since \(\lambda^2 + \mu^2 = 1\), we deduce the bounds \(-M^2 \leq \kappa \leq -m^2\),
    if \(E_{i}\) is the eigenvector of \(\operatorname{ad}(E_{N+1})\) with eigenvalue \(a_{i}\), it is not difficult to see that the sectional curvature of the 2-plane spanned by \( \{E_{i}, E_{N+1}\} \) is \(-a_{i}^2\), therefore, the previous bounds are attained and are in fact the maximum and minimum sectional curvatures.
    In the second case, there are eigenvalues \(a_{i}\) of positive and negative curvature, hence the minimum \(m\) is negative whereas the maximum \(M\) is positive and
    \(Mm \leq a_{i}a_{j} \leq \max(M^2, m^2)\), \(i, j \in \{1, \ldots, N\} \) by necessity. As in the previous case, we deduce,
    \begin{align*}
        -\max(M^2, m^2) \leq
        -\lambda^2 \sum a_{i}^2 Y_{i}^2 - \mu^2
        \sum_{i < j} a_{i}a_{j} {(X_{i}Y_{j} - X_{j}Y_{i})}^2 \leq
        - Mm.
    \end{align*}
    Suppose, without loss of generality, \(M = a_{r}\), \(m = a_{s}\) and that \(M^2 \geq m^2\), then the 2-planes spanned by \( \{E_{r}, E_{N+1}\} \) and \( \{E_{r}, E_{s}\} \) have curvatures \(-\max(M^2, m^2)\) and \(-Mm\) respectively, concluding the proof.
\end{proof}

\subsection{Logarithmic hyperbolic models of hyperbolic space}\label{sec:log-hyp-model}

Let \(\mathbb{H}^{N+1} \subset \mathbb{R}^{N+1} \) be the half-space set of points \(x\) such that \(x_{N+1} > 0\), recall the hyperbolic metric of constant curvature \(-p^2\), with \(p > 0\), is the metric on \(\mathbb{H}^{N+1}\) with line element,
\begin{equation}\label{eq:hyp-metric-curv-p2}
    ds^2_{p} = \frac{\sum_{i=1}^{N+1} dx_{i}^2}{p^2x_{N+1}^2}.
\end{equation}
We will denote the hyperbolic space of curvature \(-p^2\) as \(\mathbb{H}^{N+1}(p)\).
Define \(a = (p, \ldots, p) \in \mathbb{R}^N\), \(p > 0\), and let \(g_{a}\) be the metric~\eqref{eq:the-horospheric-metric} with parameter \(a\), we claim \((\mathbb{R}^{N+1}, g_{a})\) is isometric to \(\mathbb{H}^{N+1}(p)\), to see this, notice the map
\(\mathbb{R}^{N+1} \to \mathbb{H}^{N+1} \), such that
\(x \mapsto y\), with \(y_{i} = x_{i}, i \in \{1, \ldots, N\} \) and \(y_{N+1} = e^{p x_{N+1}}/p\), is a diffeomorphism that pulls the  hyperbolic metric~\eqref{eq:hyp-metric-curv-p2} back into \(g_{a}\). If on the other hand, \(p < 0\), the reflection with respect to the hyperplane \(x_{N+1} = 0\), pulls the metric \(g_{a}\) back into \(g_{-a}\) which is of constant negative curvature. Therefore, we have proved,
\begin{proposition}\label{prop:logarithmic-hyp-model}
    Let \(a = (p, \ldots, p) \in \mathbb{R}^{N}\), where \(p \neq 0\),
    then \((\mathbb{R}^{N+1}, g_{a})\) is isometric to hyperbolic space \(\mathbb{H}^{N+1}(|p|)\).
\end{proposition}
In particular, \((\mathbb{R}^{N+1}, g_{a})\) is a space of constant negative curvature if \(a\) is as in Proposition~\ref{prop:logarithmic-hyp-model}, in this case we call \((\mathbb{R}^{N+1}, g_{a})\) the logarithmic hyperbolic model. Now, assume without loss of generality, \(a \in \mathbb{R}^{N}\) is such that \(a_{i} > 0\) for all \(i\), but the coordinates are not necessarily equal, the case when \(a_{i} < 0\) being analogous since there is an isometry \((\mathbb{R}^{N+1},g_{a}) \to (\mathbb{R}^{N+1},g_{-a})\). Recall, if \(g_{a}\) has negative curvature, the exponential map
\(\exp: T_0\mathbb{R}^{N+1} \to \mathbb{R}^{N+1}\) is a global diffeomorphism by Hadamard's theorem.
\begin{proposition}\label{prop:exp-half-ball}
    If \(a \in \mathbb{R}^{N}\) is a vector such that \(a_{i} > 0\) for all \(i\) and
    \(v \in T_0\mathbb{R}^{N + 1}\) is such that \(v_{N+1} \leq 0\), then
    \(x = \exp(v) \in \mathbb{R}^{N+1}\) also satisfies the condition \(x_{N+1} \leq 0\).
\end{proposition}
\begin{proof}
    Let \(\gamma(s)\) be the geodesic parametrized by arc-length with initial
    conditions \(\gamma(0) = 0\), \(\dot{\gamma}(0) = v\), since \(a\) has positive coordinates, the geodesic equation for \(\ddot{\gamma}_{N+1}\) shows
    \(\dot{\gamma}_{N+1}\) is a decreasing function, since \(\gamma_{N+1}(0) = 0\)
    and \(\dot{\gamma}_{N+1}(0) \leq 0\), the result follows because \(\exp(v) =
    \gamma(1)\).
\end{proof}

Let \(\bar{D}(0, \rho)  \subset T_0\mathbb{R}^{N+1}\) be the Euclidean ball of
tangent vectors \(v\) such that \(|v| \leq \rho \), where \(|\cdot|\) is the Euclidean norm induced by \( g_{a}|_{T_0\mathbb{R}^{N+1}} \). By the Gauss Lemma, \(\exp(\bar{D}(0, \rho))\) is the closed
geodesic ball \(\overline{\mathbb{D}}(0, \rho) \subset  \mathbb{R}^{N + 1}\)
of radius \(\rho \)
centered at the origin. If \(\Pi^{-} \subset T_0\mathbb{R}^{N+1}\) is the
half-plane of vectors \(v\) such that \(v_{N+1}\leq 0\), the previous proposition
shows
\begin{equation}\label{eq:exp-geod-ball-half}
    \exp(\bar{D}(0, \rho) \cap \Pi^{-}) \subset
    \overline{\mathbb{D}}(0, \rho) \cap \mathrm{H}^{-},
\end{equation}
where \(\mathrm{H}^{-} \subset \mathbb{R}^{N+1}\) is the half-space of points \(x\) such that \(x_{N+1} \leq 0\). In the next section we use this observation to estimate the volume of geodesic balls in the negatively curved case.

\subsection{Upper bounds for the growth of geodesic balls}\label{sec:upper-bounds}

We begin with a simple estimate for the size of the coordinate \(x_{N+1}\) on points of a geodesic ball \(\overline{\mathbb{D}}(0, \rho)\) that generalize known results for \(\mathrm{SOL}\).
\begin{lemma}\label{lem:horospheric-upper-bound}
    If \(x \in \overline{\mathbb{D}}(0, \rho)\), then \(|x_{N+1}| \leq \rho \) and
    \(|x_i| \leq \rho e^{a_i\rho}\) for  \(i = 1, \ldots, N\).
\end{lemma}

\begin{proof}
    Let \(\gamma(s)\) be the unit speed geodesic joining the origin with the
    point \(x \in \overline{\mathbb{D}}(0, \rho)\), since the length of the
    tangent vector \(\dot{\gamma}\) is
    \[
        |\dot{\gamma}(s)| = {\left(\sum e^{-2a_i\gamma_{N+1}}\dot{\gamma}_i^2 + \dot{\gamma}_{N+1}^2\right)}^{1/2},
    \]
    by the triangle inequality, \(|\dot{\gamma}_{N+1}| \leq 1\) and
    \(|\dot{\gamma}_i| \leq e^{a_i\gamma_{N + 1}}\). Since \(\gamma(0) = 0\), the
    inequality \(|x_{N+1}| = |\gamma_{N+1}(\rho)| \leq \rho \) is immediate, whence
    \(|\dot{\gamma}_i| \leq e^{a_i\rho}\) and the second inequality of the Lemma follows after applying the fundamental theorem of calculus.
\end{proof}

From now onwards, we denote the Riemannian volume element of the metric \(g_{a}\) as,
\[dV = \exp\left(-\sum a_i x_{N+1}\right) dx,\]
where \(dx = dx_1\wedge \cdots \wedge dx_{N+1}\) is the canonical Euclidean volume form. 
Let \(\gamma \) be the unit speed geodesic joining the origin to the point \(x
\in \mathbb{R}^{N+1}\) at distance \(d(0,x) = \rho \), the triangle inequality
implies,
\begin{equation*}
    \int_{0}^{\rho}{(e^{-2a_i\gamma_{N+1}}\dot{\gamma}_i^2 + \dot{\gamma}_{N+1}^2)}^{1/2}
    \leq \int_{0}^{\rho} |\dot{\gamma}| = \rho,
\end{equation*}
observe that the curve \(s \mapsto (\gamma_i(s), \gamma_{N+1}(s))\) determines a path on \(\mathbb{R}^2\), whose length on the logarithmic hyperbolic model of
\(\mathbb{H}^2(|a_i|)\) is bounded by \(\rho \).
\begin{lemma}\label{lem:xi-upper-bound}
    If the geodesic distance of the point \(x \in \mathbb{R}^{N+1}\) to the origin, in the metric \(g_{a}\), satisfies \(d(0,x) \leq \rho \), and \(a_{i} > 0\), then the following inequality holds,
    \begin{align}
        |x_i| \leq \frac{\sqrt{2}}{a_i}\, e^{a_{i}x_{N+1}/2}
        {\left(
        \cosh(a_{i}\rho) - \cosh(a_{i}x_{N+1})
        \right)}^{1/2}.\label{eq:xi-upper-bound}
    \end{align}
\end{lemma}
Recall \((\mathbb{R}^2, g_{a_i})\), \(a_{i} > 0\), is isometric to the half-plane model of hyperbolic space with curvature \(-a_{i}^2\), and that the isometry is given by the map \(x + iy \mapsto x + i e^{a_{i}y}/a_{i}\).
\begin{proof}
    Consider the point \(z = (x_i, x_{N+1}) \in \mathbb{R}^2\), we know the geodesic distance from \(z\) to \(0\) with respect to the metric \(g_{a_{i}}\) is bounded by \(\rho \), let \(w = x_i + i e^{a_{i}x_{N+1}}/a_{i}\), if \(d_1\) denotes hyperbolic distance on \(\mathbb{H}^2\) and \(d_{a_i}\) is  hyperbolic distance on \(\mathbb{H}^{2}(a_i)\), then
    \[
        \frac{d_1(i/a_{i}, w)}{a_i} = d_{a_{i}}(i/{a_i}, w) = d(0, z) \leq \rho,
    \]
    whence \(w\) is in a closed geodesic ball \(\overline{D}(i/a_{i}, a_{i}\rho) \)  in the usual hyperbolic metric of \(\mathbb{H}^{2}\), since the map
    \(w \mapsto a_{i}w\) is an isometry of hyperbolic space, we deduce \(d_1(i, a_{i}w) \leq a_{i}\rho \), i.e.\  \(a_{i}w \in \overline{D}(i,a_{i}\rho)\).
    Balls of the hyperbolic plane are well characterized: the previous argument guarantees \(a_{i}w\) is within the Euclidean closed ball centred at \(i\cosh(a_{i}\rho) \) of radius \(\sinh(a_{i}\rho) \), hence,
    \[
        {(e^{a_{i}x_{N+1}} - \cosh(a_{i}\rho))}^2 + a_{i}^2x_{i}^2 \leq {\sinh(a_{i}\rho)}^2.
    \]
    From this inequality, we deduce,
    \begin{multline*}
        2e^{a_{i}x_{N+1}} \left(
        \cosh(a_{i}x_{N+1}) - \cosh(a_{i}\rho)
        \right) + a_{i}^2x_{i}^2 = \\
        e^{2a_{i}x_{N+1}} - 2\cosh(a_{i}\rho)e^{a_{i}x_{N+1}} + 1 + a_{i}^2x_{i}^2 \leq 0,
    \end{multline*}
    where we used the hyperbolic pythagorean identity. Equation~\eqref{eq:xi-upper-bound} is immediate from this inequality.
\end{proof}

\begin{proposition}\label{prop:vol-upper-bound}
    Let \(a \in \mathbb{R}^{N}\) be a vector with positive coordinates, and let \(\rho > 0\), then the volume of
    the geodesic ball \(\mathbb{D}(0, \rho) \subset \mathbb{R}^{N+1} \), with respect to the metric \(g_{a}\), is bounded,
    \begin{align}
        \operatorname{Vol}(\mathbb{D}(0, \rho)) \leq
        \frac{2^{1 + N/2}}{\sum a_{i}\cdot \prod a_{i}}\, e^{\sum a_{i}\rho}.\label{eq:volume-upper-bound}
    \end{align}
\end{proposition}
\begin{proof}
    Let \(\Omega \subset \mathbb{R}^{N+1} \) be the set of points \(x\) such that \(x_{i}\) satisfies inequality~\eqref{eq:xi-upper-bound} for
    \(i = 1, \ldots, N \) and \(|x_{N+1}| \leq \rho \). The condition on \(x_{N+1}\) implies \(\cosh(a_{i}x_{N+1}) \leq \cosh(a_{i}\rho)\), hence, by~\eqref{eq:xi-upper-bound},
    \begin{align}
        |x_i| & \leq \frac{\sqrt{2}}{a_{i}} \nonumber
        e^{a_{i}x_{N+1}/2}\cosh{(a_{i}\rho)}^{1/2} \nonumber     \\
              & = \frac{1}{a_{i}} e^{a_{i}(x_{N+1}+\rho)/2}{(1 +
        e^{-2a_{i}\rho})}^{1/2} \nonumber                        \\
              & \leq \frac{\sqrt{2}}{a_{i}}\, e^{a_{i}(x_{N+1} +
                \rho)/2}.\label{eq:xi-geod-upper-bound}
    \end{align}
    By Lemma~\ref{lem:xi-upper-bound}, \(\overline{\mathbb{D}}(0, \rho) \subset \Omega \), whence,
    % \begin{noindent}
    \begin{align*}
        \operatorname{Vol}\left(\overline{\mathbb{D}}(0,\rho)\right) & \leq \operatorname{Vol}(\Omega)                        \\
        & = \int_{\Omega} e^{-\sum a_{i} x_{N+1}} dx             \\
        & \leq \frac{ 2^{N/2}}{\prod a_{i}} e^{\sum a_{i} \rho/2}
        \int_{-\rho}^{\rho} e^{-\sum a_{i} x_{N+1}/2} dx_{N+1}\\
        &= \frac{ 2^{1+N/2}}{\sum a_{i}\prod a_{i}}
        \left(
            e^{\sum a_{i}\rho} - 1
        \right).
    \end{align*}
    % \end{noindent}
    Equation~\eqref{eq:volume-upper-bound} is immediate.
\end{proof}
As corollary to Proposition~\ref{prop:vol-upper-bound}, we obtain an upper bound for the volume entropy of \((\mathbb{R}^{N+1}, g_{a})\) when \(a\) has only positive coordinates,
\[
    \delta(g_{a}) =
    \lim_{\rho \to \infty} \frac{\log\operatorname{Vol}(\mathbb{D}(0, \rho))}{\rho} =
    \sum a_{i},
\]
where, as mentioned at the beginning, we used the fact that the space is homogeneous. In the following section, we will refine our argument to prove that this bound is in fact the volume entropy.

\subsection{Lower bounds for the growth of geodesic balls}\label{sec:lower-bound}

In this section we aim to show that the upper bound we found for the volume of
geodesic balls also works as a lower bound, this will imply the volume entropy formula for spaces of
negative curvature. Recall that as a consequence of
Proposition~\ref{prop:exp-half-ball}, the exponential map sends the lower half
of the Euclidean ball \(\bar{D}(0, \rho) \subset T_0\mathbb{R}^{N+1}\) into the
lower half of the geodesic ball \(\bar{\mathbb{D}}(0, \rho)\), as in
Equation~\eqref{eq:exp-geod-ball-half}. Let \(\Omega(\rho) =
\exp(\bar{D}(0, \rho) \cap \Pi^{-})\), we aim to show that the volume of
\(\Omega(\rho)\) grows exponentially as \(\rho \to \infty \). Also, recall \(\mathrm{H}^{-}\) denotes the lower half space of points \(x \in \mathbb{R}^{N+1}\) with \(x_{N+1} \leq 0\),
if \(\gamma: [0, \infty) \to \mathrm{H}^{-}\) is any differentiable ray,
\begin{equation}
    \int_{0}^{r} |\dot{\gamma}(s)| ds \leq
    \int_{0}^{r} {\left(
    e^{-2\sum a_{i}\gamma_{N+1}}\sum \dot{\gamma}_{i}^2 + \dot{\gamma}_{N+1}^2
    \right)}^{1/2} ds, \qquad \forall r \geq 0. \label{eq:hyp-metric-comparison}
\end{equation}
Furthermore, the integral on the right hand side is the length of the curve on the logarithmic model of hyperbolic space \(\mathbb{H}(\sum a_i)\), also,
if \(v \in S_0\mathbb{R}^{N+1} \) is such that \(v_{N+1} \leq 0\) and
\(\gamma \) is the unit speed geodesic of this hyperbolic space, by the geodesic equations, \(\gamma_{N+1}(s) \leq 0\) for  \(s \geq 0 \). Therefore, if we denote by \(\exp_{Hyp}: T_0\mathbb{R}^{N+1} \to \mathbb{R}^{N+1}\) the exponential map corresponding to the metric of \(\mathbb{H}^{N+1}(\sum a_i)\), and by \(B(\rho) = \exp_{Hyp}(\bar{D}(0, \rho)\cap \Pi^{-})\) the lower half of the hyperbolic geodesic \(\rho \)-ball in normal polar coordinates, inequality~\eqref{eq:hyp-metric-comparison} implies
\(\bar{B}(\rho) \subset \Omega(\rho)\), whence,
\[\operatorname{Vol}(\bar{B}(0, \rho)) \leq \operatorname{Vol}(\Omega(\rho)).\]
\begin{proposition}\label{prop:hyp-vol-lower-bound}
    Let \(a \in \mathbb{R}^{N}\) be such that \(a_{i} > 0\) for all \(i\), and let \(\rho > 0\), then the volume of
    the geodesic ball \(\mathbb{D}(0, \rho) \subset \mathbb{R}^{N+1} \) in the metric \(g_{a}\) is bounded
    below,
    \begin{align}
        \operatorname{Vol}(\mathbb{D}(0, \rho)) \geq
        \frac{\omega_{N}}{2 {(\sum a_{i})}^N}\,e^{-(N-1) \rho\sum a_{i}} \int_{0}^{\rho}\sinh{\left(\sum a_i r\right)}^{N}dr,\label{eq:volume-lower-bound-integral}
    \end{align}
    where \(\omega_{N}\) is volume of the round \(N\)-sphere.
\end{proposition}
\begin{proof}
    In logarithmic coordinates, the volume form of \(\mathbb{H}(\sum a_{i})\) is, 
    \[dV_{Hyp} = e^{-N \sum a_{i} x_{N+1}}dx,\]
    hence,
    % \begin{noindent}
    \begin{align*}
        \operatorname{Vol}(\bar{B}(0, \rho)) & = \int_{\bar{B}(0, \rho)} e^{-\sum a_{i}x_{N+1}} dx           \\
            & = \int_{\bar{B}(0, \rho)} e^{(N-1)\sum a_{i}x_{N+1}} dV_{Hyp}\\
            &\geq e^{-(N-1)\rho\sum a_{i}} \int_{\bar{B}(0, \rho)} dV_{Hyp}\\
            &= \frac{1}{2}e^{-(N-1)\rho\sum a_{i}} \operatorname{Vol}_{Hyp}(\mathbb{D}_{Hyp}(0, \rho)),
    \end{align*}
    % \end{noindent}
    where \(\mathbb{D}_{Hyp}(0, \rho)\) is the geodesic ball in \(\mathbb{H}(\sum a_{i})\). Let us consider the exponential map \(\exp_{Hyp}: T_0\mathbb{R}^{N+1} \to \mathbb{R}^{N+1}\), by the Gauss Lemma,
    the pull-back of the hyperbolic metric is~\cite[Pg.~135]{petersenRiemannianGeometry2016}
    \[
        dr^2 + \frac{1}{{(\sum a_i)}^2} \sinh{\left(\sum a_{i} r\right)}^2 dS_{N}^2,
    \]
    where \(dS_{N}^2\) denotes the line element of the round unit sphere \(S^{N} \subset T_0\mathbb{R}^{N+1}\). In this coordinates, the hyperbolic volume form is,
    \begin{align*}
        \operatorname{Vol}_{Hyp}(\mathbb{D}_{Hyp}(0, \rho)) & =
        \frac{\omega_{N}}{{(\sum a_{i})}^N} \int_{0}^{\rho} {\sinh\left(\sum a_{i}r\right)}^{N} dr.
    \end{align*}
    Therefore,
    \begin{align*}
        \operatorname{Vol}(\bar{B}(0, \rho)) \geq \frac{\omega_{N}}{2 {(\sum a_{i})}^N}\,e^{-(N-1) \rho\sum a_{i}} \int_{0}^{\rho} {\sinh\left(\sum a_{i}r\right)}^{N} dr,
    \end{align*}
    the proposition follows, since \(\bar{B}(0, \rho) \subset \mathbb{D}(0, \rho)\).
\end{proof}
As an immediate consequence, we obtain the following,
\begin{corollary}\label{cor:lower-bound-hyperbolic}
    For any \(b > 0\), there exist positive constants \(C_1\) and \(C_2\) such that,
    \begin{equation}
        \operatorname{Vol}(\mathbb{D}(0, \rho)) \geq C_1 e^{\sum a_{i}\rho} - C_2, \label{eq:volume-lower-bound}
    \end{equation}
    whenever \(\rho \geq b\).
\end{corollary}
\begin{proof}
    In order to simplify notation,
    let us denote by \(C_{1}, C_{2} \), etc.\ general constants whose values are not important for the proof. By Equation~\eqref{eq:volume-lower-bound-integral},
    % \begin{noindent}
    \begin{align*}
        \operatorname{Vol}(\mathbb{D}(0, \rho)) &\geq C_1 e^{-(N-1)\rho \sum a_{i}} \int_{b}^{\rho} e^{N\sum a_{i} r} {\left(
                1 - e^{-2\sum a_{i} r}
        \right)}^N dr                                           \\
         & \geq C_{2} e^{-(N-1)\rho \sum a_{i}} \int_{b}^{\rho} e^{N\sum a_{i} r} dr \\
         &= C_{3} e^{-(N-1)\rho \sum a_{i}} \left(e^{N\sum a_{i}\rho} - C_{4}\right)\\
         &= C_{3} e^{\sum a_{i} \rho} - C_{5} e^{-(N-1)\rho \sum a_{i}}\\
         &\geq C_{3} e^{\sum a_{i} \rho} - C_{6}.
    \end{align*}
    % \end{noindent}
    Renaming constant we obtain the corollary.
\end{proof}
Equations~\eqref{eq:volume-upper-bound} and~\eqref{eq:volume-lower-bound} yield that the volume entropy of \(\mathbb{R}^{N+1}\) in the metric~\eqref{eq:the-horospheric-metric} is \(\sum a_{i}\). We will use this formula in the next section to prove the general case.

\section{Volume growth of geodesic balls: General case}\label{sec:volume-growth-general-case}

In this section we prove the volume entropy formula for any space \((\mathbb{R}^{N+1}, g_{a})\). We will
find an upper bound for the volume of geodesic balls by a direct method and  extending the comparisons introduced by Schwartz for \(\mathrm{SOL}\) in~\cite{schwartzAreaGrowthSol2020}, we will find a lower bound for the volume. In this section, until further notice, we restrict the coordinates of the vector parameter \(a\) to
\(a_{i} \neq 0\) for all \(i\). Assume without loss of generality \(a_{i} > 0\) for \(i \leq M\) and
\(a_{i} < 0\) for \( i > M\). It will be convenient to define
projections \(\pi_{i}: \mathbb{R}^{N+1} \to \mathbb{R}^{N_{i} + 1}\), where
\(N_{1} = M\) and \(N_{2} = N - M\), such that \(\pi_1(x) = (x_{1},\ldots,
x_{M}, x_{N+1})\) and \(\pi_2(x) = (x_{M+1}, \ldots, x_{N+1})\).
\begin{proposition}\label{lem:vol-upper-bound-general}
    There are positive constants \(C_{1}\) and \(C_{2}\) such that
    for any \(\rho > 0\), the following bounds hold for the volume of a geodesic ball on the space \((\mathbb{R}^{N+1}, g_{a})\),
    \begin{enumerate}
        \item If \(\sum a_i = 0 \),
              \[
                  \operatorname{Vol}(\mathbb{D}(0, \rho)) \leq C_{1} \rho
                  \exp\left(\sum_{i=1}^{M}a_{i}\rho\right).
              \]
        \item Otherwise,
              \[
                  \operatorname{Vol}(\mathbb{D}(0, \rho)) \leq C_{2}
                  \left\lvert\exp\left(\sum_{i=1}^{M} a_{i}\right) -
                  \exp\left(\sum_{i=M+1}^{N}|a_i|\right)
                  \right\rvert.
              \]
    \end{enumerate}
\end{proposition}
The first case of the Proposition is the unimodular case, it is important in geometric group theory and complex Kleinian dynamics, as it can be shown that only in this case the groups \(G(A)\) described in the first section admit lattices.
\begin{proof}
    Let \(x \in \mathbb{D}(0, \rho)\) and suppose \(x_{N+1} \geq 0\), let
    \(\gamma \) be a unit speed geodesic starting at the origin and reaching
    \(x\) at time \(\rho \), the curve \(\eta = \pi_1\circ\gamma \) is a curve
    on \(\mathbb{R}^{N_{1}+1}\) joining the origin to \(y = \pi_{1}(x)\),
    moreover,
    \[
        \int_{0}^{\rho}{\left(\sum_{i=1}^{M}e^{-2a_{i}\eta_{N+1}}\dot{\eta}_{i}^2 +
        \dot{\eta}_{N+1}^{2}\right)}^{1/2} ds \leq \int_{0}^{\rho} |\dot{\gamma}|
        ds =
        \rho,
    \]
    whence, the geodesic distance from the origin to \(y\) in the hyperbolic
    metric \(g_{a'}\) of \(\mathbb{R}^{M + 1}\), with \(a' = (a_{1}, \ldots, a_{M})\), is bounded by
    \(\rho \), we can apply Lemma~\ref{lem:xi-upper-bound} exactly as in the
    proof of Proposition~\ref{prop:vol-upper-bound}, to deduce that
    equation~\eqref{eq:xi-geod-upper-bound} still holds for \(x_{i} = y_{i}\).
    Likewise, if \(z = \pi_{2}(x)\), we again obtain that geodesic distance
    from  the origin to \(z\) is bounded by \(\rho \) in \(\mathbb{R}^{N_2 + 1}\)
    with the metric of parameter \((a_{M+1}, \ldots, a_{N})\), which we denote
    by \(g^{+}\). It will be
    convenient to define \(b_i = |a_{M + i}|\), \(i = 1, \ldots, N_{2}\) and
    \(z' = f(z)\), where \(f: \mathbb{R}^{N_2 + 1} \to \mathbb{R}^{N_2 + 1}\)
    is the reflection with respect to the hyperplane \(z_{N_{2}+1} = 0\). The pull-back of \(g^{+}\) by
    \(f\)
    is a hyperbolic metric, hence, we can apply
    Lemma~\eqref{lem:xi-upper-bound} to deduce the inequality,
    \begin{equation}
        |x_{M + i}| = |z_{i}'| \leq \frac{\sqrt{2}}{b_{i}}e^{b_{i}(-x_{N+1} +
        \rho)/2}, \qquad i = 1, \ldots, N_{2}. \label{eq:xi-upper-bound-reflection}
    \end{equation}
    We proceed as in the proof of Proposition~\ref{prop:vol-upper-bound},
    however, this time we employ both equations~\eqref{eq:xi-geod-upper-bound}
    and~\eqref{eq:xi-upper-bound-reflection}, depending on the sign of each
    \(a_{i}\), \(i = 1, \ldots, N\), hence,
    \begin{align*}
        \operatorname{Vol}(\bar{\mathbb{D}}(0, \rho)) & \leq
        \operatorname{Vol}(\Omega)                                                                                         \\
                                                      & = \int_{\Omega} e^{(-\sum a_i + \sum b_j) x_{N+1}} dx              \\
                                                      & \leq \frac{2^{N/2}}{\prod_{i = 1}^{N} |a_{i}|} e^{(\sum a_i + \sum
                b_j)\rho /2} \int_{-\rho}^{\rho} e^{(-\sum_i a_i + \sum_j b_j) x_{N+1}/
                2} dx_{N+1},
    \end{align*}
    in the last inequality we simplified the product using that
    \(b_i = |a_{M + i}|\), \(i =1, \ldots, N_{2}\). There are two possibilities for the
    integral, depending on whether \(\sum a_{i} = \sum b_{j}\) or otherwise. In the first
    case,
    \[
        \operatorname{Vol}(\bar{\mathbb{D}}(0, \rho))
        \leq \frac{2\cdot 2^{N/2}}{\prod_{i = 1}^{N} |a_{i}|} \rho \exp\left({\sum
            a_{i}}\rho\right),
    \]
    which is the inequality sought in case \((1)\), whereas in the second case,
    \[
        \operatorname{Vol}(\overline{\mathbb{D}}(0, \rho)) \leq 2 C \,
        \frac{\exp(\sum a_{i}\rho) - \exp(\sum b_{j} \rho)}
        {\sum a_{i} - \sum b_{j}},
    \]
    the last assertion of the proposition follows directly from this inequality.
\end{proof}
Recall the volume entropy \(\delta(g_{a})\) is the limit~\eqref{eq:volume-entropy-limit}, since \((\mathbb{R}^{N+1}, g_{a})\) is homogeneous, we can get rid of the base point and compute the volume
entropy at the origin,
as an immediate consequence of Proposition~\ref{lem:vol-upper-bound-general}, we
can find an upper bound for the entropy.
\begin{corollary}\label{cor:upper-bound-volume-entropy}
    Let \(\delta(g_{a})\) be the volume entropy of the 
    metric~\eqref{eq:the-horospheric-metric} with parameter \(a \in
    \mathbb{\mathbb{R}}^N\) of coordinates \(a_{i} \neq 0\) for any \(i\), then the 
    following upper bound holds,
    \[\delta(g_{a}) \leq \max\left(\sum_{a_{i} > 0} a_{i},
        \sum_{a_{i} < 0}|a_i|\right).\]
\end{corollary}
In the next section we find a lower bound for the volume of geodesic balls
which will lead us to the proof of the volume entropy formula.

\subsection{A lower bound for the volume of geodesic balls}\label{sec:lower-bound-general}

Recall \(\mathrm{SOL}\) is \(\mathbb{R}^{3}\) provided with the metric \(g_{a}\) of parameter \(a = (1, -1)\). We will approximate a lower bound on the volume of geodesic balls generalizing the idea of Schwartz,
who observed that if \(\mathrm{S}^{2}(0,\rho)\) is a geodesic sphere in \(\mathrm{SOL}\), the
projection to the hyperplane \(x_{2} = 0\) is a geodesic ball on hyperbolic
space of smaller area. This observation leads to a lower estimate~\cite{schwartzAreaGrowthSol2020} of the volume of
geodesic balls in \(\mathrm{SOL}\). We follow this path in the general case, however, the main obstruction to do this is that the volume estimates rely on a precise characterization of the cut locus, which, unlike the case of \(\mathrm{SOL}\), is not available in general. We substitute this lack of knowledge with the Ito-Tanaka theorem~\cite{itohDimensionCutLocus1998}, which implies that the cut-locus has measure zero (see also~\cite[sec.~III.2]{chavelRIEMANNIANGEOMETRYModern2006} for details) and follow a measure theoretic argument.

\subsubsection{The area of a graph}
Let \(g\) be a Riemannian metric on \(\mathbb{R}^{N}\), assume that in
canonical coordinates, \(g = \sum a_{i}(x)dx_{i}^2\). Let
\(a_{N+1}: \mathbb{R}^{N} \to \mathbb{R}^{+}\) be a positive function, we
extend \(g\) to \(\mathbb{R}^{N+1}\) and define the metric
\[\tilde{g} = \sum a_{i}(x)dx_{i}^2 +
    a_{N+1}(x) dy^2.\]
Let \(f: \mathbb{R}^{N} \to \mathbb{R}\) be a differentiable function, consider
the graph \(X = \{(x, f(x)) \mid x \in \mathbb{R}^{N}\} \), together with the
parametrization \(F: \mathbb{R}^{N} \to X\), \(F(x) = (x, f(x))\).

\begin{lemma}\label{lem:comparison-volume-graph}
    For any open set \(U \subset X\),
    \(\operatorname{Vol}_{\tilde{g}}(U) \geq
    \operatorname{Vol}_{g}(F^{-1}(U))\).
\end{lemma}
\begin{proof}
    Let \(dx = dx_{1}\wedge \cdots \wedge dx_{N}\) and denote by
    \(\operatorname{dV}\), \(\operatorname{d\tilde{V}} \) the volume forms of
    \(g\) and \(\tilde{g}\) respectively, then \(\operatorname{dV} = {(\prod
    a_{i})}^{1/2}dx\), whereas,
    \[d\tilde{V} = {\left(
        \prod a_{i} + \sum
        \frac{{(\partial_{i}f)}^2}{a_{i}}
        \right)}^{1/2} dx,
    \]
    hence,
    \[
        \operatorname{Vol}_{g}(F^{-1}(U)) = \int_{F^{-1}(U)} \operatorname{dV}
        \leq \int_{F^{-1}(U)} \operatorname{d\tilde{V}} =
        \operatorname{Vol}_{\tilde{g}}(U).
    \]
\end{proof}

\begin{lemma}\label{lem:comparison-volume-submersion}
    Let \(\pi: \mathbb{R}^{N+1} \to \mathbb{R}^N\) be the canonical projection
    \(\pi(x, y) = x\), where \(x \in \mathbb{R}^N\), \(y \in \mathbb{R}\), and let
    \(X \subset \mathbb{R}^{N+1}\) be a smooth submanifold of dimension \(N\). If
    \(U \in \pi(X)\) is an open set such that any \(p \in U\) is a regular value of
    \(\pi \), then \(\operatorname{Vol}_{\tilde{g}}(\pi^{-1}(U)) \geq
    \operatorname{Vol}_{g}(U)\), where \(\operatorname{Vol}_{\tilde{g}}\) is the
    volume of the induced metric determined by the inclusion \(X \hookrightarrow
    \mathbb{R}^{N+1}\).
\end{lemma}
\begin{proof}
    Let \(p \in U\) and \(q \in \pi^{-1}(p)\), since \(p\) is regular, by the implicit function
    theorem, there are neighborhoods \(V_{p} \subset U\), \(W_{p} \subset X\)
    of \(p\) and \(q\) respectively and a function \(f_{p}: V_{p} \to
    \mathbb{R}\), such
    that the map \(F_{p}: V_{p} \to W_{p}\), \(x \mapsto (x, f_{p}(x))\) is a
    diffeomorphism. By Lemma~\ref{lem:comparison-volume-graph},
    \(\operatorname{Vol}_{\tilde{g}}(W_{p}) \geq
    \operatorname{Vol}_{g}(V_{p})\). Let \({\{\varphi_{p}\}}_{p \in U}\) be a
    partition of unity subordinate to the open cover \({\{V_{p}\}}_{p \in U}\), then, 
    \begin{align*}
        \int_{U} dV & = \sum \int_{V_p} \varphi_{p}\,dV                            \\
                    & \leq \sum \int_{V_{p}} \varphi_{p} F_{p}^{*}(d\tilde{V})     \\
                    & = \sum \int_{W_{p}} \varphi_{p}\circ F_{p}^{-1}\, d\tilde{V} \\
                    & = \int_{X} \sum \varphi_{p}\circ F_{p}^{-1}\, d\tilde{V},
    \end{align*}
    in the second equation we used Lemma~\ref{lem:comparison-volume-graph}, in
    the third one, the change of variable theorem, while for the last equation we used that
    \(\operatorname{supp}(\varphi_{p} \circ F_{p}^{-1}) \subset W_{p}\).
    Notice \( \{\varphi_{p}\circ F_{p}^{-1}\} \) is locally finite
    and \(\sum \varphi_{p}\circ F_{p}^{-1} \leq 1\), if \(x \in W_{p}\), there
    is an open neighbourhood \(N \subset U_{p}\) of \(F_{p}^{-1}(x) = \pi (x)\),
    such that, all but a finite number of the functions \(\varphi_{p_{i}}\), \(i
    \in \{1, \ldots, r\} \), are
    null, we claim \(F_{p}(N)\) is an open neighbourhood of \(x\) with the
    same property, if \(\varphi_{q}\circ F_{q}^{-1}(y) \neq 0\) for some \(y
    \in F_{p}(N)\),
    then \(\varphi_{q}(\pi(y)) = \varphi_{q}(F_{q}^{-1}(y))  \neq 0\), but
    \(\pi(y) \in N\) and therefore \(q \in \{p_1, \ldots, p_{r}\} \).
    Analogously, for any \(x \in \cup W_{p}\), 
    \[\sum_{p}\varphi_{p}\circ
    F_{p}^{-1}(x) = \sum_{p} \varphi_{p}(\pi(x)) = 1,\]
    while \(\sum_{p}\varphi_{p}\circ
    F_{p}^{-1}(x) = 0\) for \(x \not\in \cup_{p} W_{p}\). Therefore,
    \[\int_{X} \sum \varphi_{p}\circ F_{p}^{-1}\, d\tilde{V} <
        \operatorname{Vol_{\tilde{g}}}(X),\]
    and the result follows.
\end{proof}

It is clear that Lemma~\ref{lem:comparison-volume-submersion} applies whether the
projection is with respect to the last coordinate or any other of the
coordinates of \(\mathbb{R}^{N+1}\), in the following proposition, we apply the Lemma to the 
metric \(g_{a}\) and project with respect to the first coordinate, but it will be clear that the proof is independent of the projection coordinate.

\begin{proposition}\label{prop:geodesic-sphere-projection}
    Let \(\pi: \mathbb{R}^{N+2} \to \mathbb{R}^{N+1}\) be the projection
    \(\pi(x) = (x_2, \ldots, x_{N+2})\) and let \(\rho > 0\), then
    \(\pi(\mathrm{S}^{N+1}(0, \rho)) = \overline{\mathbb{D}}^{N+1}(0, \rho)\), where
    \(\mathrm{S}^{N+1}(0, \rho) \subset \mathbb{R}^{N+2}\) is the geodesic
    sphere of radius \(\rho \) in the metric \(g_{a}\) and
    \(\overline{\mathbb{D}}^{N+1}(0, \rho) \subset \mathbb{R}^{N + 1}\) is the closed geodesic ball of radius \(\rho \) in the metric \(g_{b}\), where \(b = (a_{2}, \ldots, a_{N + 2})\).
\end{proposition}

\begin{proof}
    Let \(x \in \mathrm{S}^{N+1}(0, \rho)\) and let \(\gamma: \mathbb{R} \to
    \mathbb{R}^{N+2}\) be a unit speed geodesic departing from the origin, such
    that \(\gamma(\rho) = x\), since,
    \[
        \int_{0}^{\rho} {\left(\sum_{i=2}^{N+1} e^{-2a_{i}x_{N+1}}\dot{x}_{i}^2 +
        \dot{x}_{N+2}^2\right)}^{1/2} ds \leq
        \int_{0}^{\rho} {\left(\sum_{i=1}^{N+1} e^{-2a_{i}x_{N+1}}\dot{x}_{i}^2 +
        \dot{x}_{N+2}^2\right)}^{1/2} ds,
    \]
    it is immediate that \(\pi(x) \in \overline{\mathbb{D}}^{N+1}(0, \rho)\), equivalently
    \(\pi(\mathrm{S}^{N+1}(0, \rho)) \subset \overline{\mathbb{D}}^{N+1}(0,\rho)\). On the
    other
    hand, assume \(x \in \overline{\mathbb{D}}^{N+1}(0,\rho)\), and recall the hyperplane
    \(x_{1} = 0\) is totally geodesic in \(\mathbb{R}^{N+2}\). We distinguish two cases, if \(d(0, x) =
    \rho \), then \((0,x) \in \mathrm{S}^{N+1}(0, \rho)\) and \(\pi((0,x)) = x\), otherwise, let \(r = d(0, x) < \rho \) and define the function
    \(f(s) = d(0,(s, x))\), 
    this is a continuous function such that \(f(0) < \rho \), moreover, by the
    triangle inequality,
    \[
        d((s,x), (0,x)) \leq d((s,x), 0) + d(0, (0,x)) = f(s) + d(0, (0,x)),
    \]
    since the affine subspace spanned by the canonical vectors 
    \(\{e_{1}, e_{N+2}\} \) at \(x\) is totally geodesic and
    isometric to the hyperbolic space of curvature \(|a_{1}|\),  as \(s \to \infty \),
    \[d((s,x), (0,x)) =
    d_{\mathbb{H}(|a_{1}|)}((s,x_{N+2}), (0, x_{N+2})) \to \infty,\]
    whence, there is some \(s_{0} > 0\) such that \(f(s_{0}) = \rho \), i.e.\  \( (s_{0},
    x) \in \mathrm{S}^{N+1}(0,
    \rho)\), hence, \(x \in \pi(\mathrm{S}^{N+1}(0, \rho) )\), proving the
    proposition.
\end{proof}

The proof of Proposition~\ref{prop:geodesic-sphere-projection} shows that if 
\(\mathrm{S}_{+}^{N+1}(0,\rho)\) is the right side of the
geodesic sphere, then
\(\pi(\mathrm{S}_{+}^{N+1}(0,\rho)) = \mathbb{D}^{N+1}(0, \rho)\). Let us denote by
\(\exp_{\rho} \) the restriction  of the exponential map \(\exp:
T_0\mathbb{R}^{N+2} \to \mathbb{R}^{N+2}\) to the sphere
\(\rho\mathrm{S}_0\mathbb{R}^{N+2}\) of radius \(\rho \), consider the map
\(F = \pi\circ\exp_{\rho} \), by Sard's Lemma, the set of critical values of
this map has measure zero in \(\mathbb{R}^{N+1}\) and therefore, 
the complement, i.e.\ the set of regular values, is open and dense on \(\mathbb{D}^{N+1}(0,\rho)\). Notice that, since domain and image of \(F\) have the same dimension, if \(F(v)\) is a regular value, then \(\exp_{\rho}: \rho \mathrm{S}_{0}\mathbb{R}^{N+2} \to \mathrm{S}_{+}^{N + 1}(0, \rho)\) is a local diffeomorphism in a neighborhood \(U\) of \(v\), and the restriction \(\pi|_{\exp_{\rho}(U)}\) is also a local diffeomorphism.   

\begin{proposition}\label{prop:volume-comparison-lower-bound-sphere}
    Let \(\rho > 0\), then
    \[\operatorname{Vol}(\mathbb{D}^{N+1}(0, \rho)) \leq
        \operatorname{Vol}(\mathrm{S}^{N+1}(0, \rho)).\]
\end{proposition}
\begin{proof}
    Let \(\mathrm{Reg} \subset \mathbb{D}^{N+1}(0, \rho)\) be the set of
    regular values of the map \(F = \pi\circ \exp_{\rho}:
    \rho\mathrm{S}_{0}\mathbb{R}^{N+2} \to \mathbb{R}^{N+1}\), thence
    \(\pi^{-1}(\mathrm{Reg}) \subset \mathrm{S}^{N+1}(0, \rho)\) is a smooth
    submanifold of \(\mathbb{R}^{N+2}\) open and dense in the geodesic sphere
    of radius \(\rho \), by Lemma~\ref{lem:comparison-volume-submersion},
    \[\operatorname{Vol}(\mathbb{D}^{N+1}(0, \rho)) =
        \operatorname{Vol}(\mathrm{Reg}) \leq
        \operatorname{Vol}(\pi^{-1}(\mathrm{Reg})) =
        \operatorname{Vol}(\mathrm{S}^{N+1}(0, \rho)).\]
\end{proof}
For the next lemma we will need some facts about geodesic spheres related to the cut-locus, these can be found in Appendix~\ref{sec:appendix-geodesic-spheres}.
\begin{lemma}\label{lem:disk-volume-comparison}
    Let \(a \in \mathbb{R}^{N + 1}\), \(a_{i} \neq 0\) for \(i \in \{1, \ldots, N + 1\} \), and let \(b \in \mathbb{R}^{N}\) be the vector obtained by 
    removing the first coordinate from \(a\). If \(\mathbb{D}^{N+2}(0, \rho) \subset \mathbb{R}^{N + 2}\) is the geodesic ball with respect to the metric \(g_{a}\) and \(\mathbb{D}^{N+1}(0, \rho) \subset \mathbb{R}^{N +1}\) is the geodesic ball with respect to the metric \(g_{b}\), then,
    \begin{align}
        \operatorname{Vol}(\mathbb{D}^{N+2}(0, \rho)) \geq \int_{0}^{\rho} \operatorname{Vol}(\mathbb{D}^{N+1}(0, r))\,
        dr. \label{eq:volume-comparison-full}
    \end{align}
\end{lemma}
\begin{proof}
    Let \(C \subset \mathbb{R}^{N + 2}\) be the cut locus of the metric \(g_{a}\), by Lemma~\ref{lem:sphere-singularities-measure}, for almost every \(r \in [0, \rho]\), the set \(\mathrm{S}^{N + 1}(0, r) \setminus C\) is open and dense in the geodesic sphere \(\mathrm{S}^{N + 1}(0, r)\). 
    Since \(\mathrm{S}^{N + 1}(0, \rho) \setminus C\) is a smooth manifold and \(C\) is of measure zero, we can compute the volume of the geodesic ball as,
\begin{align}
    \operatorname{Vol}\left({\mathbb{D}^{N + 2}(0, \rho)}\right) &= 
    \operatorname{Vol}\left({\mathbb{D}^{N + 2}(0, \rho) \setminus C}\right)\nonumber \\
    &= \int_{0}^{\rho} \operatorname{Vol}\left(\mathrm{S}^{N+1}(0, r)\setminus C\right) dr, \label{eq:geodesic-volume-comparison}    
\end{align}
where the last integral is computed for an open and dense subset of the interval \([0, \rho]\). Let \(\pi_{1}: \mathbb{R}^{N + 2} \to \mathbb{R}^{N + 1}\) be the projection onto the hyperplane \(x_{1} = 0\). Since \(\pi_{1}\) is differentiable, if \(\mathrm{S}^{N + 1}(0,r) \cap C\) has measure zero in \(\mathrm{S}^{N + 1}(0, r)\), then \(\pi_{1}(\mathrm{S}^{N + 1} \cap C) \) also has. By Lemma~\ref{prop:geodesic-sphere-projection}, for almost every \(r \in [0, \rho]\),  \(\pi_{1}(\mathrm{S}^{N + 1}(0, r)\setminus C)\) is dense in \(\mathbb{D}^{N + 1}(0, r) \), and by Lemma~\ref{lem:geodesic-sphere-proj-regular-value}, almost every \(x \in \mathrm{S}^{N + 1}(0, r) \setminus C\) is a regular value of the restriction  
\(\pi_{1}|_{\mathrm{S}^{N + 1}(0, \rho) \setminus C}\), whence, by 
Lemma~\ref{lem:comparison-volume-submersion}, 
\begin{align*}
    \operatorname{Vol}(\mathbb{D}^{N + 1}(0, r)) &= 
    \operatorname{Vol}(\pi_{1}(\mathrm{S}^{N + 1}(0, r) \setminus C)) \\
    &\leq \operatorname{Vol}(\mathrm{S}^{N + 1}(0, r) \setminus C).
\end{align*}
This inequality, together with Equation~\eqref{eq:geodesic-volume-comparison}, yields the Lemma.
\end{proof}

\begin{proposition}\label{prop:general-volume-upper-bound}
    Let \(a \in
    \mathbb{R}^{N+1}\) be such that \(a_i \neq 0\) for all \(i\), 
    endow \(\mathbb{R}^{N+2}\) with the Riemannian metric \(g_{a}\)
    defined as in Equation~\eqref{eq:the-horospheric-metric} and denote by 
    \(\mathbb{D}^{N + 2}(0, \rho)\) be the geodesic \(\rho \)-ball centered at the origin, 
    there are positive constants \(C\), \(D\) and \(\rho_{0}\) such that, for any
    \(\rho > \rho_{0}\),
    \begin{align}
        \operatorname{Vol}(\mathbb{D}^{N+2}(0, \rho)) \geq 
        C e^{M\rho} - D,\label{eq:volume-lower-bound-formula}
    \end{align}
    where
    \begin{equation}\label{eq:max-formula}
        M  = \max \left(
        \sum_{0 < a_{i}} a_{i}, \sum_{a_{i} < 0} |a_{i}|
        \right).
    \end{equation}
\end{proposition}

\begin{proof}
    If \(a_{i} > 0\) for all \(i\), the proposition was
    proved in Corollary~\ref{cor:lower-bound-hyperbolic}. If \(a_{i} < 0\) for all \(i\), the Proposition holds because \((\mathbb{R}^{N+1}, g_{a})\) is isometric to \((\mathbb{R}^{N+1}, g_{-a})\). Assume \(a\) has both positive and negative coordinates, we proceed by induction on \(N\). In the case \(N = 1\), the metric \(g_{a_{2}}\) of 
    \(\mathbb{R}^{2}\) is hyperbolic with curvature \(|a_2|\) and
    is well known that,
    \begin{align*}
        \operatorname{Vol}(\mathbb{D}^{2}(0, r)) &= \frac{4\pi}{|a_{2}|}
        \sinh{(|a_{2}|r/2)}^2\\
        &= \frac{\pi}{|a_{2}|} {(1 - e^{-|a_{2}|})}^2\, e^{|a_{2}| r},
    \end{align*}
    whence, by Lemma~\ref{lem:disk-volume-comparison},
    \begin{align*}%\label{eq:vol-bound-a2}
        \operatorname{Vol}(\mathbb{D}^{3}(0, \rho)) \geq
        \frac{\pi}{a_{2}^2} {(1 - e^{-|a_{2}|})}^2 (e^{|a_{2}|\rho} - 1).
    \end{align*}
    It is clear that we could permute the coordinates \(x_{1}\), \(x_{2}\) and
    also deduce the inequality,
    \begin{align*}%\label{eq:vol-bound-a1}
        \operatorname{Vol}(\mathbb{D}^{3}(0, \rho)) \geq
        \frac{\pi}{a_{1}^2} {(1 - e^{-|a_{1}|})}^2 (e^{|a_{1}|\rho} - 1).
    \end{align*}
    After choosing suitable constants \(C\) and \(D\), we conclude Equation~\eqref{eq:volume-lower-bound-formula} holds, regardless the value  of \(\rho_{0}\). 
    Suppose formula~\eqref{eq:volume-lower-bound-formula} is valid for some \(N
    \geq 2\), let
    \[
        M_{i}  = \max \left(
        \sum_{a_{j} > 0,\, j\neq i} a_{j}, \sum_{a_{j} < 0,\, <j\neq i} |a_{j}|
        \right),
    \]
    and let \(\mathbb{D}_{i}^{N + 1}(0, \rho)\) be the geodesic ball in the totally geodesic hyperplane \(x_{i} = 0\), 
    we apply the induction hypothesis and the volume comparison formula~\eqref{eq:volume-comparison-full}, in each case, there are positive constants 
    \(C_{i}\), \(D_{i}\) and \(\rho_{i}\), such that, if \(\rho > \rho_{i}\), 
    \begin{align*}
        \operatorname{Vol}(\mathbb{D}^{N+2}(0, \rho)) &\geq \int_{0}^{\rho} 
        \operatorname{Vol}(\mathbb{D}_{i}^{N + 1}(0, r))\,dr\\
        & \geq \int_{\rho_{i}}^{\rho}
        ( C_{i} e^{M_{i}r} - D_{i} )\,dr + \int_{0}^{\rho_{i}} \operatorname{Vol}(\mathbb{D}_{i}^{N + 1}(0, r))\,dr \\
        & \geq \frac{C_{i}}{M_{i}} e^{M_{i}\rho} - D_{i} \rho - E + F,\\
        &\geq \frac{C_{i}}{M_{i}} e^{M_{i}\rho} \left(
            1 - \frac{D_{i}M_{i}}{C_{i}}\,\rho e^{-M_{i}\rho}
        \right) - E,
    \end{align*}
    where \(E\) and \(F\) are positive constants whose value is not important for the proof.  Choosing a new value for \(\rho_{i}\) large enough and redefining the constants from the last step, we see that the following holds,
    \[\operatorname{Vol}(\mathbb{D}^{N+2}(0, \rho)) \geq C_{i} e^{M_{i}\rho} - D_{i}, \qquad \text{for } \rho \geq \rho_{i}.\]
    Notice that since \(a\) contains both positive and negative coordinates, for at least one \(i \in \{1, \ldots, N + 1\}\), \(M_{i} = M\), 
    concluding the induction and the proof.
\end{proof}

\subsection*{Proof of Theorem~\ref{thm:volume-entropy}}

Let \(a \in \mathbb{R}^{N}\) be a vector of parameters for the
metric~\eqref{eq:the-horospheric-metric}, since each coordinate \(a_i = 0\)
adds a flat, Euclidean factor to \(\mathbb{R}^{N+1}\) and balls in Euclidean
space grow polynomially fast, we can get rid of
such terms and assume \(a_{i} \neq 0\) for \(i \in \{i, \ldots, N\} \).
Let \(M\) be the maximum value of
Equation~\eqref{eq:max-formula}, by
Corollary~\ref{cor:upper-bound-volume-entropy}, \(\delta(g_{a}) \leq M\).
On the other hand, Proposition~\ref{prop:general-volume-upper-bound} shows
\(\delta(g_{a}) \geq M\) for \(\mathbb{R}^{N + 1}\) and \(N \geq 2\). For the
case \(N = 1\),  the metric is hyperbolic of curvature \(|a_{1}|\)
and therefore formula~\eqref{eq:volume-entropy-formula} also holds. This
concludes the
proof.

\appendix
\section{Computing the remaining of Cauchy-Schwartz}\label{sec:appendix-cauchy-schwarz}

In this section we prove the identity,
\begin{equation}\label{eq:cauchy-schwarz-residue}
    \sum a_i X_{i}^2\cdot \sum a_{i} Y_{i}^2
    - {\left(\sum a_{i} X_{i}Y_{i}\right)}^2 = \sum_{i < j} a_{i}a_{j} {(X_{i}Y_{j} - X_{j}Y_{i})}^2,
\end{equation}
used in Proposition~\ref{prop:curvature-bounds}. Although the proof can be done by brute force, we prefer a more abstract approach in the language of exterior algebra and normed vector spaces. Let \(V\) be a vector space of dimension \(N\), equipped with an inner product \(\langle \cdot, \cdot \rangle \). Recall the exterior product \(\Lambda^2V\) is the formal vector space with generators \(v\wedge w\), \(v, w \in V\), such that
\(v \wedge w = - w \wedge v\). Let us introduce an inner product in \(\Lambda^2V\) by means of the Gramian: For any pair of generators, \(v_{1} \wedge v_{2}\), \(w_{1}\wedge w_{2} \), let
\[
    \langle v_{1}\wedge v_{2}, w_{1}\wedge w_{2} \rangle = \det \begin{pmatrix}
        \langle v_{1}, w_{1} \rangle & \langle v_{1}, w_{2}\rangle \\
        \langle v_{2}, w_{1} \rangle & \langle v_{2}, w_{2}\rangle
    \end{pmatrix},
\]
and extend by linearity to the exterior product \(\Lambda^2V\). Notice that if \(v, w \in V\) is a pair of orthonormal vectors, then,
\begin{equation}\label{eq:norm-orthonormal-wedge}
    |v\wedge w|^2 = \det \begin{pmatrix}
        |v|^2                & \langle v, w\rangle \\
        \langle v, w \rangle & |w|^2
    \end{pmatrix} = 1.
\end{equation}
Assume \(A: V \to V\) is a symmetric operator, then the following identity also holds,
\begin{align}\label{eq:sym-operator-wedge-norm}
    \langle Av \wedge Aw, v \wedge w \rangle =
    \langle Av, v\rangle \langle Aw, w\rangle
    - \langle Av, w \rangle^2.
\end{align}
Let \( \{e_{1}, \ldots, e_{N}\} \subset V\) be an orthonormal basis of eigenvectors of \(A\) such that \(Ae_{i} = a_{i} e_{i}\), 
it is not difficult to see that \( \{ e_{i} \wedge e_{j}\} \), \(i < j\), is an orthonormal basis of \(\Lambda^2V\), if \(v = \sum X_i e_{i}\), \(w = \sum Y_{i} e_{i}\), a direct calculation shows
\begin{equation}\label{eq:v-wedge-w}
    v \wedge w = \sum_{i < j } (X_{i}Y_{j} - X_{j}Y_{i}) e_{i} \wedge e_{j},
\end{equation}
if \(v\) and \(w\) are orthonormal, by~\eqref{eq:norm-orthonormal-wedge},
\[
    \sum_{i < j} {(X_{i}Y_{j} - X_{j}Y_{i})}^2 = |v \wedge w|^2 = 1.
\]
Finally, another direct calculation shows
\[
    Av \wedge Aw = \sum_{i < j} a_{i}a_{j} (X_{i}Y_{j} - X_{j}Y_{i}) e_{i} \wedge e_{j},
\]
substituting this equation and~\eqref{eq:v-wedge-w} into~\eqref{eq:sym-operator-wedge-norm}, we deduce the identity~\eqref{eq:cauchy-schwarz-residue}.

\section{Geodesic spheres are measurable}\label{sec:appendix-geodesic-spheres}

In this section we review some results regarding the cut locus that we needed in the proof of Lemma~\ref{lem:disk-volume-comparison}. 
Recall \(\rho \mathrm{S}_{0}^{N+1} \subset T_0 \mathbb{R}^{N + 2}\) denotes the tangent sphere of radius \(\rho \) and \(\exp: T_{0}\mathbb{R}^{N + 2} \to \mathbb{R}^{N + 2}\) denotes the Riemannian exponential function. Let \(t_{cut}: \mathrm{S}^{N+1}_{0} \to (0, \infty]\) be the cut time function, 
\[t_{cut}(v) = \sup \{ b > 0 \mid \exp(sv) \text{ is minimizing for s } \in [0, b]\}.\]
If \(t_{cut}(V) < \infty \), the cut point of the origin along \(\exp(sv)\) is the point \(\exp(t_{cut}(v)) \in \mathbb{R}^{N + 2}\). The set of all the cut points we have defined is the Cut locus of the origin, denoted by \(\mathrm{Cut}(0)\). We also define the tangent cut locus, 
\[\mathrm{TCL}(0) = \{v \in T_{0}\mathbb{R}^{N + 2} \mid |v| = t_{cut}(v/|v|)\},\]
and the injectivity domain,
\[\mathrm{ID}(0) = \{v \in T_{0}\mathbb{R}^{N + 2} \mid |v| < t_{cut}(v/|v|)\}.\]
The following facts are well known, a proof can be found in~\cite[Ch.~10]{leeIntroductionRiemannianManifolds2018a} 
\begin{theorem}
    [Properties of the cut locus]\label{thm:cut-locus-props}
    Let \((M, g)\) be a complete, connected Riemannian manifold and \(p \in M\).
    \begin{itemize}
        \item The cut locus of \(p\) is a closed subset of \(M\) of measure zero.
        \item The restriction of \(\exp_{p}\) to \( \overline{\mathrm{ID}(p)}\) is surjective.
        \item The restriction of \(\exp_{p}\) to \(\mathrm{ID}(p)\) is a diffeomorphism onto \(M \setminus \mathrm{Cut}(p)\).
    \end{itemize}
\end{theorem}
Let us denote by \(\exp_{\rho}: \rho \mathrm{S}_{0}^{N+2} \to \mathbb{R}^{N+2}\) the restriction of the exponential map to the tangent sphere of radius \(\rho \) and let \(\pi_{i}: \mathbb{R}^{N+2} \to \Pi_{i}\) be the projection map  onto the hyperplane \(x_{i} = 0\), where \(i \in \{1, \ldots, N + 1\} \).

\begin{lemma}\label{lem:geodesic-sphere-proj-regular-value}
    If \(x \in \mathrm{S}^{N + 1}(0, \rho) \setminus \mathrm{Cut}(0)\), 
    then \(x\) is a regular value of \(\pi_{i}|_{\mathrm{S}^{N + 1}(0, \rho)}\) if and only if \(x_{i} \neq 0\).
\end{lemma}
\begin{proof} 
    By Theorem~\ref{thm:cut-locus-props}, there is a  
    \(v \in \mathrm{ID}(0)\) such that \(x = \exp(v)\) and 
    \(v\) is a regular point of the exponential map. Since \(T_{0}\mathbb{R}^{N + 2}\) is a vector space, we can identify \(v\) with a tangent vector at \(T_{v}(T_{0}\mathbb{R}^{N + 2})\), normal to \(T_{v}\rho \mathrm{S}^{N + 1}\),  
    then the Gauss Lemma implies \(d_{v}\exp(v)\) is perpendicular to \(T_{x}\mathrm{S}^{N + 1}(0, \rho) = d_{v}\exp(T_{v}\rho\mathrm{S}^{N + 1}) \). Let \(\gamma(s) = \exp(s v)\) be the radial geodesic with initial velocity \(\dot{\gamma}(0) = v\) and  
    let \(W \in T_{x}\mathrm{S}^{N + 1}(0, \rho)\), then there is a \(w \in T_{v}\rho \mathrm{S}_{0}^{N + 2}\) such that \(W = d_{v}\exp(w) \), the geodesic equations~\eqref{eq:geodesics} imply, 
    \[g_{a}(W, d_{v}\exp(v)) = \sum_{i=1}^{N + 1} W_{i}v_{i} + W_{N + 2}\dot{\gamma}_{N + 2}(1) = 0,\]
    hence, \(T_{x}\mathrm{S}^{N + 1}(0, \rho)\) is the affine hyperplane at \(x\)  that in the Euclidean metric of \(\mathbb{R}^{N + 2}\) has normal vector,
    \[n = (v_{1}, \ldots, v_{N + 1}, \dot{\gamma}_{N + 2}(1)).\] 
    Since \(\pi \) is linear, the restriction of \(d\pi \) to \(T_{x}\mathrm{S}^{N + 1}(0, \rho)\) is an isomorphism, provided \(n\) is not perpendicular to the canonical vector \(e_{i} \in \mathbb{R}^{N + 2}\), whence, \(d\pi|_{T_{x}\mathrm{S}^{N+1}(0, \rho)}\) is an isomorphism if and only if \(v_{i} \neq 0\), which, as remarked at Section~\ref{sec:geometry-of-the-metric},  holds if and only if \(x_{i} \neq 0\).
\end{proof}

\begin{lemma}\label{lem:graph-lemma}
    Suppose \(f: A \subset \mathbb{R}^{N} \to \mathbb{R}\) is continuous, \(A\) open or closed, and let \(G(f) = \{(x, f(x)) \mid x \in A\} \) be the graph of the function. If \(\Pi_{y} = \{ (x, y) \in \mathbb{R}^{N + 1} \mid x \in \mathbb{R}^{N} \} \), then the intersection \(\Pi_{y} \cap G(f)\) has measure zero for almost every \(y \in \mathbb{R}\).
\end{lemma}
\begin{proof}
    The graph of \(f\) has measure zero by~\cite[Prop.~6.3]{leeIntroductionSmoothManifolds2012}, by Fubini's theorem, \(\Pi_{y} \cap G(f)\) is measurable for almost every \(y \in \mathbb{R}\) and if \(\chi_{S}\) denotes the characteristic function of an arbitrary set \(S\), then,
    \[\int_{\mathbb{R}^{N + 1}} \chi_{G(f)} = \int_{\mathbb{R}}\int_{A} \chi_{\Pi_{y} \cap G(f)} dx\,dy = 0,\]
    therefore \(\int_{A} \chi_{\Pi_{y} \cap G(f)}\,dx = 0\) except at most for a set of measure zero.
\end{proof}
The proof of the claim of Theorem~\ref{thm:cut-locus-props} about the measure zero of the cut locus relies on the fact that the tangent cut locus is the graph of the continuous function \(t_{cut}: \mathrm{S}^{N + 1}_{0} \to [0, \infty]\), since the unit tangent sphere is compact, there are finitely many open sets \(A \subset \mathbb{R}^{N + 1}\) and smooth functions \(\phi: A \to \mathrm{S}^{N + 1}_{0}\), such that \(f = t_{cut} \circ \phi \) is a graph of the type considered in Lemma~\ref{lem:graph-lemma}.
\begin{lemma}\label{lem:sphere-singularities-measure}
    Let \(C \subset \mathbb{R}^{N + 2}\) be the cut locus of the metric \(g_{a}\), for almost every \(\rho > 0\), \(\mathrm{S}^{N + 1}_{0}(0, \rho) \setminus C\) is open and dense in \(\mathrm{S}^{N + 1}(0, \rho)\). 
\end{lemma}
\begin{proof}
    Since the cut locus is closed, the complement is relatively open in \(\mathrm{S}^{N+1}(0, \rho)\). For the second claim we use Theorem~\ref{thm:cut-locus-props}, note that \(\mathrm{S}^{N + 1}(0, \rho) \cap C = \exp(G(t_{cut}) \cap \rho \mathrm{S}^{N + 1}_{0})\), if \(\phi: A \to \mathrm{S}^{N + 1}_{0}\) is a local coordinate system and \(f = t_{cut} \circ \phi \), then, in the notation of Lemma~\ref{lem:graph-lemma},
    \[G(t_{cut}|_{\phi(A)}) \cap \rho \mathrm{S}^{N + 1}_{0} = G(f) \cap \Pi_{\rho},\]
    whence, \(G(t_{cut}|_{\phi(A)}) \cap \rho \mathrm{S}^{N + 1}_{0}\) has measure zero for a.e.\  \(\rho \), since the unit tangent sphere is compact, it is covered by a finite number of coordinates, hence, \(G(t_{cut}) \cap \rho \mathrm{S}^{N + 1}_{0}\) also has measure zero for a.e.\  \( \rho \). Finally, the exponential map is differentiable, hence \(\exp(G(t_{cut}) \cap \rho \mathrm{S}^{N+1}_{0})\) also has measure zero for a.e.\  \(\rho \), which is equivalent to the second claim.
\end{proof}

\section*{Acknowledgement}

The author acknowledges that this work was supported by UNAM Postdoctoral Program (POSDOC).

\bibliographystyle{plain}
\bibliography{ComplexKleinian}
\end{document}